\documentclass[12pt]{amsart}
\usepackage{amsmath,amssymb,amsbsy,amsfonts,amsthm,latexsym,
                        amsopn,amstext,amsxtra,euscript,amscd,mathrsfs,color,bm,cite}
                       
\usepackage{float} 
\usepackage[english]{babel}
\usepackage{mathtools}
\usepackage{todonotes}
\usepackage{url}
\usepackage[colorlinks,linkcolor=blue,anchorcolor=blue,citecolor=blue,backref=page]{hyperref}
\usepackage{stackengine}
\newcommand{\Dbar}{\stackinset{l}{0.1ex}{c}{}{\rule{0.33em}{0.3pt}}{D}}

\usepackage{enumitem}

\def\le{\leqslant}
\def\ge{\geqslant}
\def\leq{\leqslant}
\def\geq{\geqslant}

\def\Case#1#2{%
\smallskip\paragraph{\textbf{\boldmath Case #1: #2}}\hfil\break\ignorespaces}

\def\Subsubcase#1#2{%
\smallskip\paragraph{\textit{$\bullet$ Subcase #1: #2}}\hfil\break\ignorespaces}

\newtheorem{theorem}{Theorem}

\newtheorem{lemma}[theorem]{Lemma}

\newcommand{\Z}{\mathbb{Z}}

\newcommand{\R}{\mathbb{R}}

\newcommand{\N}{\mathbb{N}}

\newcommand{\e}{\operatorname{e}}

\usepackage{etoolbox}

\numberwithin{equation}{section}
\numberwithin{theorem}{section}

\newcommand{\QQ}{\mathbb{Q}}
\newcommand{\Fq}{\mathbb{F}_q}

\def\cG{{\mathcal G}}

\def\cI{{\mathcal I}}
\def\cJ{{\mathcal J}}

\def\cR{{\mathcal R}}
\def\cS{{\mathcal S}}

\def\fA{{\mathfrak A}}
\def\fB{{\mathfrak B}}
\def\fS{{\mathfrak S}}

\def \balpha{\bm{\alpha}}
\def \bbeta{\bm{\beta}}
\def \bgamma{\bm{\gamma}}

\def\ov\QQ{\overline{\QQ}}

\def\e{\mathbf{e}}
\def\eq{{\mathbf{\,e}}_q}
\def\em{{\mathbf{\,e}}_m}

\def\mand{\qquad\mbox{and}\qquad}

\usepackage{mathtools}
\usepackage{todonotes}
\usepackage[norefs,nocites]{refcheck}

\def\({\left(}
\def\){\right)}

\def\rf#1{\left\lceil#1\right\rceil}

\begin{document}

\title[Distribution of modular square roots of primes] 
{On the distribution of modular square roots of primes}

 \author[I. D. Shkredov]{Ilya D. Shkredov}
\address{I.D.S: Steklov Mathematical Institute of Russian Academy
of Sciences, ul. Gubkina 8, Moscow, Russia, 119991; \  Institute for Information Transmission Problems  of Russian Academy
of Sciences, Bolshoy Karet\-ny Per. 19, Moscow, Russia, 127994; \ 
Moscow Institute of Physics and Technology, Institutskii per. 9, Dolgoprudnii, Russia, 141701}
\email{ilya.shkredov@gmail.com}

 \author[I.~E.~Shparlinski]{Igor E. Shparlinski}
 \address{I.E.S.: School of Mathematics and Statistics, University of New South Wales.
 Sydney, NSW 2052, Australia}
 \email{igor.shparlinski@unsw.edu.au}
 
 \author[A. Zaharescu]{Alexandru Zaharescu} 
 \address{A.Z.: Department of Mathematics, University of Illinois at
Urbana-Cham\-paign  1409 West Green Street, Urbana, IL 61801, USA;   \ 
Simon Stoilow Institute of Mathematics of the Romanian Academy, P.O. Box 1-764, RO-014700 Bucharest, Romania}
 \email{zaharesc@illinois.edu}

 \begin{abstract}  
We use recent bounds on bilinear sums with modular square roots to study the 
distribution of solutions to congruences $x^2 \equiv p \pmod q$ with primes $p\le P$ and integer $q \le Q$. 
This can be considered as a  combined scenario 
of Duke, Friedlander and Iwaniec with averaging only over the modulus  $q$ and of Dunn, Kerr, Shparlinski and   Zaharescu 
with averaging only over $p$. \end{abstract}


\maketitle
\tableofcontents

\section{Introduction}

\subsection{Motivation} 

We recall that the celebrated work of  Duke, Friedlander and Iwaniec~\cite{DFI1, DFI2}, see also~\cite{Hom, Toth},  establishes 
 the uniformity of distribution of fractions 
$x(n,q) /q$ formed by all 
solutions to quadratic congruence 
\begin{equation}
\label{eq:cong n}
x^2 \equiv n \pmod q, \qquad 1 \le x \le q\,, 
\end{equation}
for a given integer $n$ and the prime modulus $q$ that
runs up to some bound $q \le Q$. These results have had an enormous number of applications, see
for example~\cite{AhDu, Cil, Dok, DFI2, LiMas, Masr, RSZ}. 

In~\cite{DKSZ} a somewhat dual question has been considered
about the distribution of $x(p,q) /q$ for a fixed prime $q$ when $p$ runs over primes $p\le P$ 
for some parameter $P$, with non-trivial estimates provided that $P \ge q^{2/3+ \varepsilon}$ with some fixed $\varepsilon > 0$.

Here we consider a combined scenario of congruences $x^2 \equiv p \pmod q$ when $p$  varies over primes $p \le P$
and $q$ varies over integers $q \le Q$. 

More precislely, given a prime $q$ and a real parameter $P$ we consider the set $\cR_q(P)$ of 
primes $p \le P$ which are quadratic residues modulo $q$. Following~\cite{DKSZ}, we are
interested in the distribution of solutions to the congruence
$$
x^2 \equiv p \pmod q, \qquad  p \in \cR_q(P) \,. 
$$
Obviously to be able to answer this question one needs good lower bounds on the abundance 
of primes in  $\cR_q(P)$, that is, on the cardinality
$$
R_q(P) = \# \cR_q(P)\,.
$$
Unfortunately, unless $P$ is exponentially large, all known results of this type are conditional on the Generalised Riemann Hypotheis or 
other conjectures on the zero-free regions of $L$-functions, see~\cite{DKSZ}.

Here we  show that a result of~\cite{DKSZ} on  square roots of primes in residue rings 
modulo  $q$ can be improved on average over $q$ and can also be given in a fully unconditional form.
This is based on two  ingredients:
\begin{itemize}
\item  an asymptotic formula for $R_q(P)$ on average over $q$ which follows from a large sieve-type 
result of Heath-Brown~\cite{H-B1} on average values of sums of real characters;

\item  a new bound of bilinear sums with modular square roots of integers from~\cite{SSZ} 
which we couple with the  Heath-Brown identity (see~\cite[Proposition~13.3]{IwKow}) to estimate 
exponential sums with  square roots of primes.
\end{itemize}

As we have noticed, our result is an unconditional  averaged version of a result from~\cite{DKSZ} with averaging 
over the modulus $q$.
It can also be viewed  as an averaged version of  results  of  Duke, Friedlander and Iwaniec~\cite{DFI1, DFI2}, Homma~\cite{Hom}
and~\cite{Toth}, to the scenario when $n$ in~\eqref{eq:cong n}   varies over primes $p \le P$. 

Finally, we mention that the results and methods of~\cite{LSZ1,LSZ2} undoubtedly lead to 
a uniformity of distribution result for solutions to 
 \begin{equation} 
 \label{eq:Mod qk}
x^2 \equiv p \pmod {q^k}, \qquad p \in \cR_q(P)\,, 
\end{equation} 
with a fixed
odd prime $q$ as $k\to \infty$, starting with very short intervals, namely, already for $P\ge q^{\varepsilon k}$
with any fixed $\varepsilon > 0$ (clearly $p \in \cR_q(P)$ is equivalent to the solvability of~\eqref{eq:Mod qk}).

\subsection{New result}
More precisely,   given $\lambda \in \Z_q^\times$,  where $\Z_q^{\times}$ is the unit 
group of the residue ring $\Z_q$ modulo $q$, and two real numbers, we define
$$
\Delta_{\lambda, q} (P)  = \max_{[Y+1, Y+X]\in [1, q-1]}  \left|T_{\lambda, q}(P;X,Y) - \frac{X}{q}  \pi(P)\right| \,,
$$
where a $T_{\lambda, q}(P;X,Y)$ denotes the number of $x \in [Y+1, Y+X]$ with $x^2 \equiv  \lambda p \pmod q$ 
for some prime $p \le P$ and, as usual,  $\pi(P)$ denotes the number of primes $p \le P$. 

In~\cite{DKSZ} the discrepancy $\Delta_{\lambda, q}(P)$ is estimated under the condition that for the given prime $q$, 
the number of prime quadratic 
residues $p \le P$ is close to its expected value $0.5 \pi(P)$. 
Here we take   advantage of averaging over $q~\text{prime}$ and obtain an unconditional 
result with a stronger bound on average. 

\begin{theorem}
\label{thm:disc p Q}  Let  $1\le P \le Q$. Then we have 
$$
\frac{1}{Q} \sum_{\substack{q \le Q\\q~\text{prime}}} \max_{\lambda \in \Z_q^\times} \Delta_{\lambda, q}(P) \le   \(P^{11/12}  +  P^{4/5}  Q^{1/10}\)Q^{o(1)} \,.
$$ 
\end{theorem} 

It is easy to see that Theorem~\ref{thm:disc p Q}  is nontrivial  for $P \ge Q^{1/2+ \varepsilon}$ with some fixed $\varepsilon > 0$, 
while for $P=Q$ we get $\Delta_{\lambda, q} (q)  \le    q^{11/12  + o(1)}$, for almost all primes $q$, 
uniformly over $\lambda \in \Z_q^\times$.   

Perhaps considering more cases in the proof of Theorem~\ref{thm:disc p Q} one can obtain a improve the  bound
of Theorem~\ref{thm:disc p Q}.
However our goal has been to have a nontrivial result in a range of $P$ as wide as possible and we believe that 
the above condition $P \ge Q^{1/2+ \varepsilon}$  is the limit of our method. 

Clearly, using  Theorem~\ref{thm:disc p Q} one can provide averaged versions of many applications which rely on  the bound
of Duke, Friedlander and Iwaniec~\cite[Theorem~1.1]{DFI2}; some of them are indicated already in~\cite{DFI2}, some other can 
be found in~\cite{AhDu,  Dok, LiMas, Masr}.

\section{Links to other problems} 

\subsection{Local spacings}
The local spacing distribution of the sequence $n^2 \alpha$ mod 1 for $\alpha$ irrational has been extensively studied in the literature. A classical result of Rudnick and Sarnak states that for all integers $d \ge 2$ and almost all real $\alpha$, the pair correlation of the sequence $n^d \alpha$ mod 1 is Poissonian. This is in contrast with the case $d = 1$, where it is well known 
that for all $\alpha$ and all $N$, the gaps between consecutive elements of fractional parts $\{n\alpha\}$, 
$1 \le n \le N$, can take at most three values.
Returning to the case $d = 2$, Rudnick, Sarnak and Zaharescu~\cite{RuSaZ, Z} have
shown that for sufficiently well approximable numbers $\alpha$, the $m$-level correlations and consecutive spacings are Poissonian along subsequences. For $\alpha = \sqrt{2}$, these types of conjectures are supported 
numerically~\cite{GCI} because of their close connection to the distribution between neighbouring levels of a certain integrable quantum system.

A difficult problem is that of the distribution of local spacings
between consecutive primes. 
Gallagher~\cite{Ga} proves that the sequence of primes has a Poissonian distribution, conditionally under the
assumption of (a uniform version of) an even more famous conjecture, the prime $k$-tuple conjecture.

Let us now take a large prime 
number $q$ and consider two sequences modulo $q$: the sequence of primes up to $q$, and the sequence of
squares of positive integers up to $N$. Suppose $N$ is of the size of $q/\log q$, so that the above two finite
sequences have about the same number of elements. By Gallagher's result~\cite{Ga}, the first sequence 
has a Poissonian distribution,
conditionally under the prime k-tuple conjecture. Unconditionally, by~\cite{RuSaZ}, the second sequence has a Poissonian 
distribution for $N$ of the above size.
Under these circumstances one would naturally expect that if one takes the union of these two sequences, the new
sequence has a Poissonian distribution, too. Thus, for example, the nearest-neighbor distribution should be
exponential: for each fixed $\lambda > 0$, the proportion of gaps between consecutive elements of the sequence
(arranged increasingly in the interval $[1, q]$) should tend to $e^{-\lambda}$, as $q$ tends to infinity. Note that the
distribution problem for this combined sequence introduces new challenges. Thus, if one wants to count
neighbours (pairs of consecutive elements of the sequence) asymptotically, one needs to deal with four types of
pairs: pairs where both elements are primes (counted in \cite{Ga}), pairs where both elements are squares mod $q$
(counted in \cite{RuSaZ}),
 as well as new types of pairs, where one element is a prime and the other is a square. 
 Counting these new types of pairs leads one to study the problem of finding, for each fixed integer $h$, an
 asymptotic formula for the number of solutions to the 
 congruence $n^2 \equiv p + h$ (mod $q$). Here the case $h = 0$ would need to be included, too, and in that case 
 the problem reduces to the congruence discussed in the present paper.

\subsection{Diophantine inequalities}
Diophantine inequalities with primes and respectively with squares have a long history. In the case of 
primes, Matom\"aki \cite{Ma} proved that for
any real irrational number $\alpha$, and any $\varepsilon > 0$,  there are infinitely many prime numbers $p$ for which
$$
\|p\alpha\| <\frac{1}{p^{1/3 - \varepsilon}} \,.
$$
where $\|\xi\|$ denotes the distance between $\xi$ and the closest integer. 

In the case of squares, it is shown in~\cite{Z1} that for any real irrational number $\alpha$ and 
any $\varepsilon > 0$,  there are infinitely many positive integers $n$ for which
$$
\|n^2\alpha\| <\frac{1}{n^{2/3 - \varepsilon}} \,.
$$

The following question naturally arises:
Given a real irrational number $\alpha$ and positive 
integers $P$ and $N$, can one find a prime $p \le P$ and a positive 
integer $n \le N$ such that $p\alpha$ and $n^2\alpha$ are close to each other modulo 1? 

Here one may expect that since $n$ can take $N$ values and $p$ can take about $P/\log P$ values, there
should be a pair $(p, n)$ for which the distance between the fractional part of $p\alpha$ and the
fractional part of $n^2\alpha$ is less than $1/(PN)^{1-\varepsilon}$. Such an expectation is simply false.

Indeed, consider for instance the case $P = N^2$. Then all the differences $p - n^2$ that can appear are
nonzero integers in the interval $[-P, P]$. Recall that Dirichlet's theorem is best possible: almost all real numbers have
{\it Diophantine type\/} exactly 2. For such an $\alpha$, one cannot find nonzero integers $m$ in the
interval $[-P, P]$ for which $\|m\alpha\| < 1/P^{1+\varepsilon}$, and 
therefore one cannot find a pair $(p, n)$ as above for which the distance between the fractional part of $p\alpha$ and the
fractional part of $n^2\alpha$ is less than $1/P^{1+\varepsilon}$.

We remark that for the same real numbers  $\alpha$, (that is, of Diophantine type equal to 2) one can combine~\cite{Ma} with~\cite{Z1} to
conclude that for infinitely many $P = N^2$ as above, there exist pairs $(p, n)$ for which 
$$\|p\alpha - n^2\alpha\| < 1/P^{1/3-\varepsilon}.
$$

To obtain the result one actually makes both $\|p\alpha\|$  and $\|n^2\alpha\|$ smaller than $1/P^{1/3-\varepsilon}.$ This
applies in particular to the case when the given real irrational number $\alpha$ is algebraic, by the {\it Thue--Siegel--Roth theorem\/}.

Let us remark that the above type of questions have connections with some celebrated
unsolved problems involving primes and squares. For example, a well known conjecture of Hardy and Littlewood states
that every large enough positive integer is either a square, or the sum of a prime and a square. Assuming this
holds true, and applying it to $m$ above (or applying it to $2m$ in case $m$ is a square), it 
follows that there is a pair $(p, n)$ such that the fractional part of $p\alpha$ and the 
fractional part of $n^2\alpha$ are either both $O(1/P)$, or both are $1 - O(1/P)$, or they are at distance $O(1/P)$ from being
symmetrically placed with respect to 1/2. 

Another well known 
conjecture of Hardy and Littlewood states that any large enough odd number is the sum of a prime and 2 times a square.
Assuming this conjecture holds true, and applying it to $2m+1$ in a similar way as above, it 
follows that there is a pair $(p, n)$ such that $p\alpha$ and $2n^2\alpha$ are at distance $O(1/P)$ from being
symmetrically placed with respect to $\alpha/2$ modulo 1. Less famous than the celebrated Goldbach conjecture, this
conjecture actually goes back to Goldbach, too. He stated the
conjecture in a letter to Euler dated 18 November 1752. For more on the
history of this problem, the reader is referred to Hodges \cite{Ho}.

Suppose now that $\alpha$ has a higher Diophantine type, and let $b/q$ be a rational number such that
 \begin{equation} 
 \label{eq:type K}
\left| \alpha - \frac{b}{q}\right | < \frac{1}{q^K} \, .
\end{equation} 
Assume $K > 3$. Also, assume that both $P$ and $N$ are smaller than $q$, and do not necessarily 
satisfy $P = N^2$. Now if one tries to find a prime $p$ up to $P$ and a positive
integer $n$ up to $N$ such that $p\alpha$ is close to $n^2\alpha$ modulo 1, then one is
actually forced to restrict themselves to only consider pairs $(p, n)$ for which $n^2 \equiv p \pmod q$.
Indeed, for any other pair $(p, n)$ the numbers $bp$ and $bn^2$ are  incongruent modulo $q$,
so $pb/q$ and $n^2b/q$  differ by at least $1/q$. On the other hand 
$$|p\alpha - pb/q| < P/ q^K < 1/q^{K-1}$$
and similarly 
$$|n^2\alpha - n^2b/q| < N^2/q^K < 1/q^{K-2}\,.
$$ 

With $K > 3$, both the above quantities are much 
smaller than $1/q$. Thus 
$$
\|p\alpha - n^2\alpha\| \gg 1/q\, .
$$

By contrast, each pair $(p, n)$  for which $n^2 \equiv p \pmod q$ automatically
produces a better result :
$$
\|p\alpha - n^2\alpha\| \ll \frac{P+N^2}{q^K} \, .
$$

We end this subsection with the following remark.
Notice that one may be able to improve on this bound by studying the 
distribution of square roots of primes $p$ up to $P$ modulo $q$. This is
directly related to the topic of the present paper. Indeed, a strong bound on the 
discrepancy of such a set of square roots would imply the existence of
such square roots in reasonably short intervals. In particular, it would imply the
existence of numbers $n \le N$, with $N$ reasonably smaller than $q$, with $n^2$
congruent mod $q$ to a prime less than $P$. This is 
achieved in Theorem~\ref{thm:disc p Q}  above, not for every $q$, but for most primes $q$ up to $Q$.
There is however no principal obstacle to extending this result to averaging over all integers $q \le Q$. 
Furthermore, using standard tools of the theory of Diophantine approximations, it is easy to show that for any $K>0$, for a set of $\alpha \in [0,1]$ of 
positive Hausdorff dimension, there are infinitely many approximations~\eqref{eq:type K} with primes $q$.

\section{Preliminaries} 

\subsection{Notation} 

Throughout the paper, the notation $U = O(V)$, 
$U \ll V$ and $ V\gg U$  are equivalent to $|U|\leqslant c V$ for some positive constant $c$, 
which throughout the paper may depend on a small real positive parameter $\varepsilon$. 

For any quantity $V> 1$ we write $U = V^{o(1)}$ (as $V \to \infty$) to indicate a function of $V$ which 
satisfies $|U| \le V^{\varepsilon}$ for any $\varepsilon> 0$, provided $V$ is large enough.

For a sequence of complex weights  $\bgamma = (\gamma_k)_{k=1}^K$ 
and $sigma>0$, we denote 
$$
\| \gamma \|_{\infty}=\max_{k=1, \ldots, K} |\gamma_k| \quad \text{and} \quad \| \gamma \|_{\sigma}= \(\sum_{k=1}^K |\gamma_k|^{\sigma} \)^{\frac{1}{\sigma}} \,.
$$

For a real $A> 0$, we write $a \sim A$ to indicate  that $a$ is in the dyadic interval $A \le  a < 2A$. 

For $\xi \in \R$ and $m \in \N$ we denote
$$\e(\xi)=\exp(2 \pi i \xi) \quad \text{and} \quad \em(\xi) = \exp(2 \pi i \xi/m) \,.
$$
We also use $(k/q)$ to denote the Jacobi symbol of $k$ modulo an odd integer $q\ge 2$.  

We always use the letter $p$, with or without subscript, to denote a prime number.

As usual, for an integer $a$ with $\gcd(a,q)=1$ we define $\overline{a}$ by the conditions
$$
a \overline{a}  \equiv 1 \pmod q \mand \overline{a} \in \{1, \ldots, q-1\} \,.
$$

We also use $\mathbf{1}_{\cS}$ to denote the characteristic function of a set $\cS$ and denote by $|\cS|$ 
the cardinality of this set.
Finally, we recall that 
$$
\sideset{}{^*} \sum \mand \sideset{}{^\sharp} \sum 
$$
mean that the summation is over elements of $\Z_q^\times$ and over odd integers, respectively.

\subsection{Bilinear forms and equidistribution}
Given  $a,h \in \Z_q^{\times}$,   integer numbers $M,N\ge 1$ and complex weights 
$$
\balpha = (\alpha_m)_{m=1}^M  \quad \text{and} \quad \bbeta =  (\beta_n)_{n=1}^N \,,
$$
we consider bilinear forms in Weyl sums for square roots
\begin{equation} \label{Wdef}
W_{a,q}(\balpha, \bbeta; h,M,N) =    \sideset{}{^*}\sum_{m=1}^M   \sideset{}{^*}\sum_{n=1}^N \alpha_{m} \beta_{n} \sum_{\substack{x \in \Z_q \\
x^2 = amn}} \eq(hx) \,, 
\end{equation}
where, as mentioned above, $\sum^*$ means that the summation is over the elements of $ \Z_q^{\times}$.
We also  remark that the equation $x^2 = amn$ in the definition of the sums~\eqref{Wdef} 
is considered in $\Z_q$ and thus is equivalent to the congruence $x^2 \equiv amn \pmod q$. 

The goal is to  improve the trivial bound 
$$ 
W_{a,q}(\balpha, \bbeta; h,M,N) = O\(  \| \balpha \|_1   \| \bbeta \|_1\)\,.
$$
For many applications this is especially important ti achieve below the so-called {\it P{\'o}lya--Vinogradov range\/}, that is, for $M,N \le q^{1/2}$, since 
as  it has been shown by Dunn and Zaharescu~\cite{DuZa}  this leads to a  power saving in 
the error term of an asymptotic formula for a second moment of certain $L$-functions.  

For prime $q$, first nontrivial bounds on  the sums~\eqref{Wdef}  have been given in~\cite{DuZa} and then improved in~\cite[Theorem~1.7]{DKSZ},
as follows 
 \begin{equation} 
 \begin{split}
 \label{eq:Old-1}
 |W_{a,q}(\balpha,\bbeta;h,M,N)|\le  \|\balpha\|_2& \|\bbeta\|_{\infty}^{1/3}  \|\bbeta\|_1^{2/3}   q^{1/8+o(1)} M^{7/24}N^{1/8}\\
 & \qquad\left(\frac{M^{7/48}}{q^{1/16}}+1\right)\left(\frac{N^{7/48}}{q^{1/16}}+1\right)
  \end{split}
\end{equation} 
and
 \begin{equation} 
 \begin{split}
 \label{eq:Old-2}
|W_{a,q}(\balpha, \bbeta; h,M,N)|   \le \| \balpha \|_2  &\| \bbeta \|_1^{3/4}   \| \bbeta \|_{\infty}^{1/4} q^{1/8+o(1)} M^{5/16} N^{1/16}\\
 & \qquad  \(\frac{M^{3/16}}{q^{1/8}}  +1\)  \(\frac{N^{3/16}}{q^{1/8}} +1 \)\,.
  \end{split}
\end{equation} 

Furthermore, in~\cite{SSZ} the bounds~\eqref{eq:Old-1}  and~\eqref{eq:Old-2} have been improved on average over $q$ where 
the averaging involves all odd integers $q$ rather than only primes.

First we note that without loss of generality the weights $\balpha$ and $\bbeta$ 
can be normalized to satisfy 
\begin{equation} \label{eq:norm}
 \| \balpha \|_2 \le M^{1/2}  \mand  \| \bbeta \|_{\infty} \le 1 \,.
 \end{equation} 
 In particular, under the condition~\eqref{eq:norm}, the  bounds~\eqref{eq:Old-1}  and~\eqref{eq:Old-2} become
  \begin{equation} 
 \begin{split}
 \label{eq:OldSimple-1}
|W_{a,q}&(\balpha, \bbeta; h,M,N)| \\ &\le  q^{1/8+o(1)}   (MN)^{19/24}
 \left(\frac{M^{7/48}}{q^{1/16}}+1\right)\left(\frac{N^{7/48}}{q^{1/16}}+1\right) 
 \end{split}
\end{equation}
and
 \begin{equation} 
 \begin{split}
 \label{eq:OldSimple-2}
|W_{a,q}&(\balpha, \bbeta; h,M,N)| \\
  &\le  q^{1/8+o(1)}   (MN)^{13/16}     \(\frac{M^{3/16}}{q^{1/8}}  +1\)  \(\frac{N^{3/16}}{q^{1/8}} +1 \)\,,
 \end{split}
\end{equation} 
respectively.

Recall that for real positive $Q$ we write $q \sim Q$ to indicate $q \in [Q, 2Q)$. 
It is convenient to define 
 \begin{equation} 
 \begin{split}
 \label{eq:BMNQ}
\fB(M,N,Q) =   (MN)^{3/4} Q^{1/8}   &  \(M^{1/4} Q^{-1/8}     + 1\) \\
& \qquad \(N^{1/4} Q^{-1/8} + 1\) \,,
\end{split} 
\end{equation} 
which is  the bound of~\cite{SSZ} on the sums $W_{a,q}(\balpha, \bbeta; h,M,N)$ on average.

More precisely, we now consider the average
value
\begin{align*}
\fA(Q)   = \frac{1}{Q} \sideset{}{^\sharp} \sum_{q \sim Q}    \max_{1\le M,N \le Q}  \ &
\max_{a,h \in \Z_q^{\times}}\\& \quad  \max_{\balpha, \bbeta \ \text{as in~\eqref{eq:norm}}} \( \frac{ \left |W_{a,q}(\balpha, \bbeta; h,M,N)\right|} {\fB(M,N,Q)} \)^4 \,,
\end {align*}
where, as before,  $\sum^\sharp$ means that the summation is over odd integers. 

By a result of~\cite{SSZ}, we have 

\begin{lemma} \label{lem:Waq-Aver}
For  $Q \to \infty $, 
we have 
$$
\fA(Q)= Q^{o(1)} \,.
$$
\end{lemma}

In particular, if $M, N \le Q^{1/2}$ then the bound $\fB(M,N,Q)$ in~\eqref{eq:BMNQ}  takes form 
$$
\fB(M,N,Q)  \ll   (MN)^{3/4} Q^{1/8} \,.
$$
This is better than the trivial bound  
provided that $MN \ge Q^{1/2+\varepsilon}$,  for some 
fixed $\varepsilon > 0$, 
while the bounds~\eqref{eq:OldSimple-1}  and~\eqref{eq:OldSimple-2}   
of~\cite{DKSZ} require $MN \ge Q^{3/5+\varepsilon}$
and  $MN \ge Q^{2/3+\varepsilon}$, respectively.

\subsection{Distribution of prime quadratic residues on average}
\label{sec:dist av} 

We use the following immediate implication of a  result of Heath-Brown~\cite{H-B1}
on the overage values of sums of real characters.   

Let $N_q(P)$ be the number of primes $p\le P$ which are quadratic residues modulo $q$.

\begin{lemma}
\label{lem:AvRealChar}
Let  $1\le P \le Q$. Then we have 
$$
\frac{1}{Q}  \sum_{\substack{q \le Q\\q~\mathrm{prime}}} \ \left| N_q(P) - \frac{1}{2} \pi(P)\right| \le  P^{1/2} Q^{o(1)} \,.  
$$
\end{lemma} 

\begin{proof}
Clearly, 
$$
 N_q(P)  =  \frac{1}{2}    \sum_{p \le P}\ \(\(\frac{p}{q}\)+1\) \,. 
 $$
 Hence 
\begin{equation} \label{eq:NqP}
 \sum_{\substack{q \le Q\\q~\textrm{prime}}} \left| N_q(P) - \frac{1}{2} \pi(P)\right|  \le \sum_{q \le Q} \left|   \sum_{p \le P}\(\frac{p}{q}\)\right| \,.
\end{equation}
On the other hand, a very special case  of~\cite[Theorem~1]{H-B1}
implies 
\begin{equation} \label{eq:H-B Bound}
 \sum_{\substack{q \le Q\\q~\textrm{prime}}} \ \left|   \sum_{p \le P}\(\frac{p}{q}\)\right|^2 \le PQ^{1+o(1)} \,.
\end{equation}
Using  the Cauchy inequality, we see that~\eqref{eq:NqP}  and~\eqref{eq:H-B Bound} imply the desired 
result. 
\end{proof}

\subsection{Exponential sums and discrepancy}

 We recall that  the {\it discrepancy\/} $D(N) $ of a sequence in $\xi_1, \ldots, \xi_N \in [0,1)$    is defined as 
\begin{equation}
\label{eq:Discr}
D_N = \frac{1}{N}\sup_{0\le \gamma \le 1} \left |  \#\{1\le n\le N:~\xi_n\in [0, \gamma)\} - \gamma N \right | \,.
\end{equation} 
We remark that this notion of discrepancy is normalized by the presence of the factor $1/N$. One may also work
with the unnormalized discrepancy, where the factor $1/N$ is missing from the right side of~\eqref{eq:Discr}.
Thus the normalized discrepancy is bounded by 1, the unnormalized discrepancy is bounded by $N$, and
the connection between them is simply that the unnormalized discrepancy equals $N$ times the 
normalized discrepancy.   

We now recall  the classical Erd\H{o}s--Tur\'{a}n inequality which links the discrepancy and 
exponential sums (see, for instance,~\cite[Theorem~1.21]{DrTi} or~\cite[Theorem~2.5]{KuNi}).

\begin{lemma}
\label{lem:ET}
Let $x_n$, $n\in \N$,  be a sequence in $[0,1)$. Then for any $H\in \N$, the discrepancy $D_N$ given by~\eqref{eq:Discr} satisfies
$$
D_N \le 3 \left( \frac{1}{H+1} + \frac{1}{N}\sum_{h=1}^{H}\frac{1}{h} \left| \sum_{n=1}^{N} \e(h\xi_n) \right | \right) \,.
$$
\end{lemma}

It  is now  useful to recall the definition of the {\it Gauss sum\/} 
$$
\cG_q(a,b,q)= \sum_{x \in \Z_q} \e_q \(ax^2+bx \), \qquad (a,b) \in\Z_q^\times  \times \Z_q \,.
$$
The standard evaluation~\cite[Theorem~3.3]{IwKow},  for odd 
integer  modulus $q \ge 3$ leads to the formula
\begin{equation} \label{eval}
\cG_q(a,b)=\e_q \(-\overline{4a} b^2 \) \varepsilon_{q} \sqrt{q} \( \frac{a}{q} \) \,, 
\end{equation} 
where 
$$
 \varepsilon_{q} =\begin{cases}
1 \quad & \text{if } q\equiv 1 \pmod 4\,,\\
i  \quad & \text{if } q\equiv -1 \pmod 4 \,.
\end{cases}
$$

We also need the following bound for exponential sums over square roots modulo primes.
Since below $q$ is always prime,  we use the notation of the finite field $\Fq$ of $q$ elements 
instead of $\Z_q$. 

\begin{lemma}
\label{lem:Incom Sqrt}
For a prime $q$, an integer $W \le q$  and 
integers $a$ and $h$ with $\gcd(ah,q)=1$, we have 
$$
 \sum_{w = 1}^W \sum_{\substack{x \in \Fq \\ x^2= aw}} \e_q(hx) \ll q^{1/2+o(1)}\,.
$$
\end{lemma} 

\begin{proof} Completing the exponential sum  as in~\cite[Section~12.2]{IwKow} gives that
$$
 \sum_{w = 1}^W\sum_{\substack{x \in \Fq \\ x^2=a w}} \e_q(hx) 
  \ll \max_{0 \leq t \leq q-1} \left|\cG_q(t \overline{a},h) \right | \cdot  \log q \,. 
$$
The value $t=0$ does not contribute anything  and for any $t \in\Fq^\times $ use~\eqref{eval}.   
\end{proof}


\section{Proof of Theorem~\ref{thm:disc p Q}}

\subsection{Preliminary discussion} 
We follow closely the approach of~\cite{DKSZ}, however our estimates are slightly 
different, so we present the proof in full detail.  We recall that  in Theorem~\ref{thm:disc p Q}
the modulus $q$ runs through primes. Hence we use $\Fq$ instead of $\Z_q$. 

As in~\cite{DKSZ}, we see that  Lemma~\ref{lem:ET} reduces the    discrepancy 
question  to 
estimating the  exponential sum
$$
S_{q}(h,P)= \sum_{p\le P} \sum_{\substack{x \in \Fq \\ x^2 = p}} \eq(hx)  $$
for $P \le q$. 
Thus our goal is to estimate 
$$
\fS(P,Q) =  \sum_{\substack{q \sim Q\\q~\text{prime}}} \max_{h \in \Fq^\times} \left| S_{q}(h,P)\right| \,. 
$$

In turn,  using partial summation,  one can bound the sums $\fS(T,Q)$ via the sums
\begin{equation}
 \label{eq:SPQ-Lambda}
\widetilde{\fS}(T,Q) =  \sum_{\substack{q \sim Q\\q~\text{prime}}} \max_{h \in  \Fq^\times} \left|  \widetilde{S}_{q}(h,T)\right| \,,
\end{equation}
with $R \le T$, where 
$$
\widetilde{S}_{q}(h,T) = \sum_{k=1}^T \Lambda(k) \sum_{\substack{x \in \Fq \\ x^2 =k}} \eq(hx)\\
$$
and,   as usual, we use
$$
\Lambda(n) =
\begin{cases}
\log p&\quad\text{if $n$ is a power of the prime $p$}\,,\\
 0 &\quad\text{otherwise}\,, 
\end{cases}
$$
to denote the {\it von Mangoldt function\/}.  

Thus our goal is to derive the estimate 
\begin{equation}
 \label{eq:SSPQ-Bound}
\widetilde{\fS}(P,Q) \le   \(P^{11/12}  +  P^{4/5}  Q^{1/10}\)Q^{o(1)} \, .
\end{equation} 

In what follows it is convenient to define 
$$
\rho_q  = 
\max_{1\le M,N \le Q} \
\max_{a,h \in \Z_q^{\times}}\,  
\max_{\balpha, \bbeta \ \text{as in~\eqref{eq:norm}}}\  
\max\left\{1, \frac{ \left |W_{a,q}(\balpha, \bbeta; h,M,N)\right|} {\fB(M,N,Q)} \right\} \,. 
$$

Hence, recalling~\eqref{eq:BMNQ}, and using $(MN)^{1/4} Q^{-1/4} \le P^{1/4}  Q^{-1/4} \ll 1$ we see that for $q \sim Q$ we can always use the bound
\begin{equation}
\begin{split} 
 \label{eq:W rhoq}
&  \left |W_{a,q}(\balpha, \bbeta; h,M,N)\right| \\
& \qquad  \le 
\rho_q    (MN)^{3/4}   Q^{1/8+o(1)}     \(M^{1/4} Q^{-1/8}     + 1\)  \(N^{1/4} Q^{-1/8} + 1\) \\
& \qquad   \le \rho_q    (MN)^{3/4}   Q^{1/8+o(1)}     \(M^{1/4} Q^{-1/8}   + N^{1/4} Q^{-1/8} + 1\) \,.   
\end{split} 
\end{equation}

Our main tool is the bound
\begin{equation}
 \label{eq:sum rhoq}
\frac{1}{Q} \sum_{\substack{q \sim Q\\q~\text{prime}}} \rho_q  \le Q^{o(1)}
\end{equation} 
implied by  Lemma~\ref{lem:Waq-Aver} and the H{\"o}lder  inequality, combined with~\eqref{eq:W rhoq}. 

\subsection{The Heath-Brown identity} 

To estimate the sum~\eqref{eq:SPQ-Lambda} we apply the Heath-Brown identity in the form  given 
by~\cite[Lemma~4.1]{FKM}
(see also~\cite[Proposition~13.3]{IwKow})   as well as a smooth partition of 
unity from~\cite[Lemme~2]{Fouv} (or~\cite[Lemma~4.3]{FKM}).

We also fix three  parameters 
\begin{equation}
 \label{def:UL}
U \ge S \ge  L \ge 1\,.
\end{equation}
 to be optimised later
and define 
\begin{equation}
 \label{def:J}
J = \rf{\log P/\log L} \,.
\end{equation}
We always assume that $L$ exceeds some fixed small power of $q$ so we always have $J \ll 1$.

Now, as in~\cite[Lemma~4.3]{FKM}, we decompose  $\widetilde{S}_q(h,P)$ into a linear combination of $O(\log^{2J} q)$ sums
with coefficients bounded by $O(\log q)$,
\begin{align*}
\Sigma_q(\mathbf{V})=\sum_{m_1, \ldots, m_J=1}^{\infty} &\gamma_1(m_1)\ldots  \gamma_J(m_J)  \sum_{n_1,\ldots , n_J=1}^\infty \\
\\& V_1 \( \frac{n_1}{N_1} \)  \ldots V_J \( \frac{n_J}{N_J} \)
\sum_{\substack{x \in \Fq \\ x^2=m_1 \ldots m_J n_1 \ldots  n_J}} \e_q(hx) \, ,
\end{align*}  
where  
\begin{equation} \label{eq:cond1}
\mathbf{V}=(M_1,\ldots , M_J,N_1,\ldots  ,N_J) \in [1/2,2P]^{2J} 
\end{equation}
is a $2J$-tuple of parameters satisfying
 \begin{equation} \label{eq:cond2}
N_1 \geq \ldots \ge N_J, \quad  M_1,\ldots ,M_J \leq P^{1/J},\quad   P \ll  R  \ll  P 
\end{equation}
(implied constants are allowed to depend on $J$),  with 
 \begin{equation} \label{eq:prod R}
R =  \prod_{i=1}^J M_i \prod_{j =1}^JN_j\,,
\end{equation}  
and
\begin{itemize}
\item the arithmetic functions $m_i \mapsto \gamma_i(m_i)$ are bounded and supported in $[M_i/2,2M_i]$;
\item the smooth functions $x_i \mapsto V_i(x)$ have support in $[1/2,2]$ and for any fixed $\varepsilon> 0$ satisfy
$$
V^{(j)}(x) \ll  q^{j \varepsilon}
$$
for all integers $j \geq 0$, where the implied constant may depend on $j$ and $\varepsilon$. 
\end{itemize} 

We recall that  the notation $a \sim A$ is   equivalent  to $a \in [A/2, 2A)$. Hence we 
can rewrite he sum $\Sigma_q(\mathbf{V})$ in the following form
\begin{align*}
\Sigma_q(\mathbf{V})= \sum_{\substack{m_i \sim M_i,  n_i \sim N_i\\ i =1, \ldots J}}  \gamma_1(m_1)\ldots  \gamma_J(m_J) 
  V_1& \( \frac{n_1}{N_1} \)  \ldots V_J \( \frac{n_J}{N_J} \) \\ 
& \quad 
\sum_{\substack{x \in \Fq \\ x^2=m_1 \ldots m_J n_1 \ldots  n_J}} \e_q(hx)\,. 
\end{align*}  
In particular, we see that the sums $\Sigma_q(\mathbf{V})$ are supported on a finite set. 
We now collect various bounds on the sums $\Sigma_q(\mathbf{V})$ which we derive 
in various ranges of parameters $M_1,\ldots , M_J,N_1,\ldots  ,N_J$ until we cover the whole
range in~\eqref{eq:cond1}.

\subsection{Bounds of multilinear sums} 
To estimate the multilinear sums $\Sigma_q(\mathbf{V})$, we put $N_1$ in ranges which we call ``{\it{small\/}}'', ``{\it{moderate\/}}", ``{\it{large\/}}" and  ``{\it{huge\/}}''.
We further split the ``{\it{moderate\/}}" range in further subranges depending on
``{\it{small\/}}" and ``{\it{large\/}}" values of $N_2$.  
These ranges depend on $L$, $S$ and $U$ in~\eqref{def:UL} and also $P$ and $Q$ and thus in principle  some can be empty depending 
on the choice of $L$, $S$ and $U$.

In order to apply Lemma~\ref{lem:Waq-Aver},  it is convenient to observe that in the  
bound~\eqref{eq:W rhoq} we have
$M^{1/4}Q^{-1/8} + 1 \ll 1$ for $M \ll Q^{1/2}$ and similarly for the other term involving $N$. 
It is also convenient to assume that 
\begin{equation}
 \label{eq: P large}
P\ge Q^{1/2},
\end{equation} 
as otherwise the bound of Theorem~\ref{thm:disc p Q} is trivial.  
 
\Case{I}{Small $N_1$} 
First we consider the case when 
\begin{equation}
 \label{eq: small N1}
 N_1  \le L \,.
\end{equation}

From the definition of $J$ in~\eqref{def:J} and the condition~\eqref{eq:cond2}  we see that 
\begin{equation}
 \label{eq: small M}
M_1,\ldots ,M_J \leq  L \,.
\end{equation}

We see that if~\eqref{eq: small N1} holds then we 
can  choose two arbitrary sets $\cI, \cJ \subseteq\{1, \ldots, J\}$ such that for 
$$
M = \prod_{i\in \cI} M_i \prod_{j \in \cJ} N_j \mand N = R/M \,, 
$$
where $R$ is given by~\eqref{eq:prod R} we have 
 \begin{equation}
 \label{eq: N q12}
P^{1/2} \ll N \ll L^{1/2}P^{1/2} \,. 
\end{equation}
Indeed, we simply start multiplying consecutive elements of the sequence $M_1,\ldots , M_J,N_1,\ldots  ,N_J$
until their product $R_+$ exceeds $P^{1/2}$ while the previous product $R_- < P^{1/2}$. 
Since by~\eqref{eq: small N1} and~\eqref{eq: small M} each factor is at most $L$, we have $R_+ < L R_-$.  
Hence  
\begin{itemize}
\item either we have $P^{1/2} \le R_+ \le P^{1/2}L^{1/2}$ and then we set 
 $M= R/R_+$ and $N = R_+ $;
\item  or we have $P^{1/2} >R_- > L^{-1/2} P^{1/2}$ and then we set 
$M = R_-$ and $N= R/R_-$, where $R$ is given by~\eqref{eq:prod R}. 
\end{itemize}
Hence in either case  the corresponding  $N$ satisfies the 
upper bound in~\eqref{eq: N q12}.  
In this case, since for $N \gg P^{1/2}$ we have $M \ll P/N \ll  P^{1/2} \ll Q^{1/2}$,
recalling~\eqref{eq:W rhoq}, we have 
\begin{equation}
\begin{split}
 \label{eq: Bound1}
|\Sigma_q(\mathbf{V})| & \le \rho_q  P^{3/4}    Q^{1/8+o(1)}   \(L^{1/8} P^{1/8}Q^{-1/8}+1\) \\
& = \rho_q  \(L^{1/8}  P^{7/8}     + P^{3/4}    Q^{1/8}\)Q^{o(1)}  \,. 
\end{split}
\end{equation}

 \Case{II}{Moderate $N_1$} 
 We now  consider the case 
 \begin{equation}
 \label{eq: med N1}
L < N_1  \le S \,.
\end{equation}
where we now assume that 
 \begin{equation}
 \label{eq: small S}
S \le P^{1/2}\,. 
\end{equation}
We further split it into two subcases, depending on the size of $N_2$.

 \Subsubcase{II.1}{Moderate $N_1$ and small $N_2$} 
If we have 
$$
S  \ge  N_1\ge  L > N_2 \,, 
$$
then we again start multiplying $N_1$ by other elements from  the sequence $M_1,\ldots , M_J,N_2,\ldots  ,N_J$
and using that~\eqref{eq: P large} guarantees that in this range, using~\eqref{eq: small S}, we have
$$
 \prod_{i=1}^J M_i \prod_{j =2}^JN_j = R/N_1 \ge  P/S \ge P^{1/2}
 $$
we can  prepare some integers $M$ and $N$ with~\eqref{eq: N q12}.
Therefore, we  again have the bound~\eqref{eq: Bound1}.  

 \Subsubcase{II.2}{Moderate $N_1$   and $N_2$} 
 It remains to consider the case when 
$$
S \ge  N_1\ge N_2 \ge L \,.
$$

In this case, we define $M$ and $N$ as
$$
M =  \prod_{i=1}^J M_i \prod_{j =3}^JN_j \mand   N =  N_1 N_2
$$
thus we have 
$$
S^2 \ge  N \ge L^2 \,.
$$
We also note that 
\begin{align*}  
 M^{1/4} Q^{-1/8}  + N^{1/4} Q^{-1/8}   
 &  \ll (P/N) ^{1/4} Q^{-1/8}  + N^{1/4} Q^{-1/8} \\
 &  \ll (P/L^2) ^{1/4} Q^{-1/8}  + Q^{-1/8} S^{1/2}  \\
 &  =  L^{-1/2} P^{1/4} Q^{-1/8}  +  Q^{-1/8} S^{1/2}   \,.
 \end{align*} 
Hence,   the bound~\eqref{eq:W rhoq},  implies
\begin{equation}
\begin{split} 
 \label{eq: Bound22a}
|\Sigma_q(\mathbf{V})| &  \le  \rho_q    P^{3/4}   Q^{1/8}   \(L^{-1/2} P^{1/4} Q^{-1/8}+ Q^{-1/8}  S^{1/2} +1\)\\
& =  \rho_q\(L^{-1/2} P  +   P^{3/4}S^{1/2}+   P^{3/4}   Q^{1/8}\)  Q^{o(1)} \,. 
\end{split} 
\end{equation}
We now observe that~\eqref{eq: Bound22a} is trivial when~\eqref{eq: small S} fails, so do 
not have to restrict $S$ to~\eqref{eq: small S} anymore.

 \Case{III}{Large $N_1$} 
In the case when 
  \begin{equation}
 \label{eq: large N1}
U \ge  N_1 \ge S
\end{equation}   
we set 
$$
M = \prod_{i=1}^J M_i \prod_{j =2}^JN_j \mand N =   N_1 \,.
$$
With the above choice, under the condition~\eqref{eq: large N1}, we  have the bounds
\begin{align*}  
 M^{1/4} Q^{-1/8}  + N^{1/4} Q^{-1/8}   
 &  \ll (P/N) ^{1/4} Q^{-1/8}  + N^{1/4} Q^{-1/8}  \\
 &  \ll (P/S ) ^{1/4} Q^{-1/8}  + U^{1/4} Q^{-1/8} \\
 &  =  P^{1/4} Q^{-1/8} S^{-1/4} + U^{1/4} Q^{-1/8}  \,.
 \end{align*} 
 Therefore, we see that the bound~\eqref{eq:W rhoq}  implies
\begin{equation}
\begin{split} 
 \label{eq: Bound3}
|\Sigma_q(\mathbf{V})| & \le    \rho_q    P^{3/4}  Q^{1/8+o(1)}  \(P^{1/4} Q^{-1/8} S^{-1/4} + U^{1/4} Q^{-1/8} + 1\)  \\
 & \le    \rho_q   \(  P  S^{-1/4}  +  P^{3/4} U^{1/4}   + P^{3/4}  Q^{1/8}\)    Q^{o(1)} \,.
\end{split} 
\end{equation}

\Case{IV}{Huge $N_1$} 
 We now consider the case  when
  \begin{equation}
 \label{eq: huge N1}
U  <N_1 \le  P \,.
\end{equation}  
In this case, via partial summation  and an application of  Lemma~\ref{lem:Incom Sqrt},  exactly as
in~\cite{DKSZ} 
\begin{equation}
\begin{split} 
 \label{eq: Bound4}
|\Sigma_q(\mathbf{V})| & \le M_1 \ldots  M_J N_2 \ldots  N_J  q^{1/2+o(1)}  \\
&= P  N_1^{-1}   q^{1/2+o(1)} \le  \rho_qP  U^{-1}   Q^{1/2+o(1)} \,.
\end{split} 
\end{equation}

\subsection{Optimisation} 
We observe that  the bounds~\eqref{eq: Bound1},   \eqref{eq: Bound22a},  
 \eqref{eq: Bound3}  
and~\eqref{eq: Bound4} cover all four possiblee ranges of $N_1$ given 
by~\eqref{eq: small N1}, \eqref{eq: med N1},  \eqref{eq: large N1} and~\eqref{eq: huge N1}. 

We now choose   $L$  to balance its contribution to  
the bounds~\eqref{eq: Bound1} and~\eqref{eq: Bound22a}.
This leads us to the equation
$$ 
L^{1/8}  P^{7/8}  =    L^{-1/2} P \,.
 $$
Thus, we choose
$$
L= P^{1/5} \,, 
$$
in which case the bound~\eqref{eq: Bound22a} always dominatess~\eqref{eq: Bound1} (as it has one extra term) and hence both can be 
combined as 
\begin{equation}
 \label{eq: BoundOpt1}
|\Sigma_q(\mathbf{V})|  \le  \rho_q  \(P^{9/10} +  P^{3/4}S^{1/2}+  P^{3/4}    Q^{1/8}\)Q^{o(1)}  \,. 
\end{equation}

We also  have $J \ll 1$ as required. 

We now choose $S$  balance its contribution to  
the bounds~\eqref{eq: Bound3} and~\eqref{eq: BoundOpt1}. That is, we chose it as $S= P^{1/3}$ 
from the equation 
$$
 P  S^{-1/4} = P^{3/4}S^{1/2}\,.
 $$
 Hence the  bounds~\eqref{eq: Bound3} and~\eqref{eq: BoundOpt1} (after discarding the term $P^{11/12}$) can now be combined as 
 \begin{equation}
 \label{eq: BoundOpt2}
|\Sigma_q(\mathbf{V})|  \le  \rho_q  \(P^{11/12}+   P^{3/4} U^{1/4}  + P^{3/4}    Q^{1/8}\)Q^{o(1)}  \,. 
\end{equation}

We also choose $U$ to balance its contribution to  
the bounds~\eqref{eq: Bound4} and~\eqref{eq: BoundOpt2}.
This leads us to the equation
$$ 
P  U^{-1}   Q^{1/2} = P^{3/4} U^{1/4} \,.
$$
Thus, we choose
$$
U= P^{1/5} Q^{2/5} \,,
$$
in which case the bound~\eqref{eq: BoundOpt2} dominates~\eqref{eq: Bound4}  and both can be combined as 
\begin{equation}
 \label{eq: BoundOpt3}
|\Sigma_q(\mathbf{V})|  \le  \rho_q   \(P^{11/12} +  P^{4/5}  Q^{1/10}  + P^{3/4}  Q^{1/8}\)    Q^{o(1)}    \,. 
\end{equation}

We now note that under the condition~\eqref{eq: P large} we have
$$
 P^{4/5}  Q^{1/10}  \ge  P^{3/4}  Q^{1/8} \,.
 $$
Therefore the bound~\eqref{eq: BoundOpt3}  
implies that 
\begin{equation}
 \label{eq: BoundOpt4}
|\Sigma_q(\mathbf{V})|  \le  \rho_q  \(P^{11/12}  +  P^{4/5}  Q^{1/10} \)Q^{o(1)} 
\end{equation}
for all  data $\mathbf{V})$ and hence we obtain~\eqref{eq:SSPQ-Bound}.

\subsection{Concluding the proof} 

Combining the bound~\eqref{eq:SSPQ-Bound} with~\eqref{eq:sum rhoq} and Lemma~\ref{lem:ET}, 
we get
\begin{equation}
\begin{split} 
 \label{eq: Discr}
\frac{1}{Q} \sum_{\substack{q \sim Q\\q~\text{prime}}}  
\max_{[Y+1, Y+X]\in [1, q-1]}  & \left|T_{\lambda, q}(P;X,Y) - \frac{2X}{q}  N_q(P) \right| \\
& \quad   \le   \(P^{11/12}  +  P^{4/5}  Q^{1/10}\)Q^{o(1)} \,,
\end{split} 
\end{equation}
where $N_q(P)$ is as in Section~\ref{sec:dist av}, that is,  the number of primes $p\le P$ 
which are quadratic residues modulo $q$. Note that each prime counted by $N_q(P)$ contributes 
two values of $x$. 

Together with  Lemma~\ref{lem:AvRealChar},  the bound~\eqref{eq: Discr} concludes the proof.

\section{Possible generalisations}

It is natural to ask whether our results and methods can be used to treat higher degree roots of primes, that is, 
to ask about the distribution of roots of congruences
$$
x^d \equiv p \pmod q, \qquad  p \in \cR_q(P) \,, 
$$
with an integer $d \ge 3$.

To address this question, we recall that one of the crucial ingredients in the proof of Theorem~\ref{thm:disc p Q} 
is a result of Heath-Brown~\cite{H-B1} on average values of sums of real characters. 
Similar, albeit weaker, results are also known for cubic and quartic characters, see~\cite{BaYo,GaZh,H-B2},
however we are unaware of any result for higher order characters. This can limit the abilities of what one can 
realistically hope to prove nowadays to $d \le 4$, unless one assumes the Generalised Riemann Hypothesis, 
which instantly gives such a necessary result for each $q$ (without any need for averaging), see~\cite[Sections~5.8 and~5.9]{IwKow}. 

The second ingredient is provided by bounds of bilinear sums with roots which in turn is based on
bounds on the {\it additive energy\/} of roots. The case of square-roots allows a special treatment, see~\cite{DKSZ,SSZ}, 
however higher degree roots can be studied as well. To illustrate this we consider the congruence with cubic roots 
\begin{equation}
 \label{eq: Cube}
x+y \equiv a \pmod q, \qquad  x^3, y^3 \in [1,N]\,, 
\end{equation}
where the cubes are computed modulo $q$.  From~\eqref{eq: Cube} we derive 
$$
a^3 \equiv (x+y)^3 = x^3 + y^3 + 3axy \pmod q
$$
and then 
$$
27a^3 x^3y^3  \equiv \(a^3 - x^3 - y^3\)^3 \pmod q\,.
$$
Denoting $U = x^3y^3$ and $V = x^3 + y^3$ we arrive to the congruence 
$$
\(a^3 - V\)^3   \equiv  27a^3 U  \pmod q
$$
with $V \in [1, 2N]$, $U \in [1, N^2]$ to which, provided $N^2 = o(p)$, the methods of~\cite{Chang,CCGHSZ}
can be applied. Quite to the contrary to above limitation $ d\le 4$, we believe that this part can 
be extended to arbitrary $d \ge 2$.

\section*{Acknowledgement} 

The work of I.E.S. was supported in part  by the Australian Research Council Grant~DP170100786. 


%

\end{document}